\documentclass[11pt]{amsart}
\usepackage{etex}
\usepackage[margin=1in]{geometry}
\usepackage{amsfonts,amsmath,amssymb,amsthm,amscd,amsxtra}
\usepackage{mathtools,mathptmx}
\usepackage{enumerate,epsfig,verbatim}
\usepackage{centernot}
\usepackage{todonotes}
\usepackage[T1]{fontenc}
\usepackage[usenames,dvipsnames]{pstricks}
\usepackage[mathscr]{eucal}
\usepackage[notcite,notref]{}
\usepackage{tikz}
\usetikzlibrary{graphs,graphs.standard,matrix, calc, arrows,quotes}

\usepackage[all,2cell,ps]{xy}
\usepackage[pagebackref]{hyperref}
\usepackage{nicefrac}
\usepackage{array}
\usepackage{xcolor}
\usepackage{color}
\usepackage[all]{xypic}

\usepackage[notcite,notref]{}
\usepackage{url}
\usepackage{datetime}

\SelectTips{cm}{}

\def\latex/{{\protect\LaTeX}}
\def\latexe/{{\protect\LaTeXe}}
\def\amslatex/{{\protect\AmS-\protect\LaTeX}}
\def\tex/{{\protect\TeX}}
\def\amstex/{{\protect\AmS-\protect\TeX}}
\def\bibtex/{{Bib\protect\TeX}}
\def\makeindx/{\textit{MakeIndex}}

\theoremstyle{plain} 

\newtheorem{thm}{Theorem}[section]
\newtheorem{prop}[thm]{Proposition}
\newtheorem{cor}[thm]{Corollary}

\theoremstyle{definition}
\newtheorem{chunk}[thm]{\hspace*{-1.065ex}\bf}
\newtheorem{lem}[thm]{Lemma}

\newtheorem{eg}[thm]{Example}

\newtheorem{ques}[thm]{Question}

\theoremstyle{remark}

\newtheorem*{claim*}{Claim}

\newcommand{\ZZ}{\mathbb{Z}}

\newcommand{\QQ}{\mathbb{Q}}

\newcommand{\fn}{\mathfrak{n}}
\newcommand{\fm}{\mathfrak{m}}
\newcommand{\fp}{\mathfrak{p}}

\newcommand{\ltensor}{\tensor^{\bf{L}}}
\newcommand{\RHom}{\operatorname{RHom}}

\newcommand{\tensor}{\otimes}

\DeclareMathOperator{\ann}{ann}
 
\DeclareMathOperator{\Tor}{Tor}
\DeclareMathOperator{\Ext}{Ext}

\DeclareMathOperator{\Hom}{Hom}

\DeclareMathOperator{\Tr}{\textnormal{Tr}}

\DeclareMathOperator{\G}{G}

\DeclareMathOperator{\Ass}{Ass}

\DeclareMathOperator{\Spec}{Spec}

\DeclareMathOperator{\pd}{pd}

\DeclareMathOperator{\depth}{depth}

\DeclareMathOperator{\Ker}{Ker}

\newcommand{\bb}{\left[ \begin{smallmatrix}}
\newcommand{\eb}{\end{smallmatrix} \right]}

\def\urltilda{\kern -.15em\lower .7ex\hbox{\~{}}\kern
  .04em}\def\urldot{\kern -.10em.\kern -.10em}\def\urlhttp{http\kern
  -.10em\lower -.1ex\hbox{:}\kern -.12em\lower 0ex\hbox{/}\kern
  -.18em\lower 0ex\hbox{/}} 

\begin{document}

\title[On the vanishing of self extensions]{On the vanishing of self extensions of even-periodic modules}

\author[E. Celikbas]{Ela Celikbas}
\address{Ela Celikbas\\
School of Mathematical and Data Sciences\\
West Virginia University\\
Morgantown, WV 26506-6310, U.S.A}
\email{ela.celikbas@math.wvu.edu}

\author[O. Celikbas]{Olgur Celikbas}
\address{Olgur Celikbas\\
School of Mathematical and Data Sciences\\
West Virginia University\\
Morgantown, WV 26506-6310, U.S.A}
\email{olgur.celikbas@math.wvu.edu}

\author[Hiroki Matsui]{Hiroki Matsui}
\address{Hiroki Matsui\\ Department of Mathematical Sciences,
Faculty of Science and Technology,
Tokushima University,
2-1 Minamijosanjima-cho, Tokushima 770-8506, JAPAN}
\email{hmatsui@tokushima-u.ac.jp}

\author[R.\ Takahashi]{Ryo Takahashi}
\address{Ryo Takahashi\\
Graduate School of Mathematics, Nagoya University, Furocho, Chikusaku, Nagoya, Aichi 464-8602, Japan}
\email{takahashi@math.nagoya-u.ac.jp}
\urladdr{https://www.math.nagoya-u.ac.jp/~takahashi/}

\thanks{Matsui was partly supported by JSPS Grant-in-Aid for Early-Career Scientists 22K13894}
\thanks{Takahashi was partly supported by JSPS Grant-in-Aid for Scientific Research 23K03070}

\maketitle

\setcounter{tocdepth}{1}

\begin{abstract} In this paper we study rigid modules over commutative Noetherian local rings, establish new freeness criteria for certain periodic rigid modules, and extend several results from the literature. Along the way, we prove general Ext vanishing results over Cohen-Macaulay rings and investigate modules which have zero class in the reduced Grothendieck group with rational coefficients. 
\end{abstract}

\section{Introduction}

Throughout $R$ denotes a commutative Noetherian local ring with maximal ideal $\fm$ and residue field $k$, and all $R$-modules are assumed to be finitely generated.

An $R$-module $M$ is called \emph{rigid} if $\Ext^1_R(M,M)=0$. The classification of rigid modules has been of interest in both commutative algebra and representation theory; see, for example \cite{HSW, OY}. Work by Auslander \cite{Au}, Jothilingham \cite{Jot}, and Lichtenbaum \cite{Li} shows that, over regular rings, the only rigid modules are free modules; see also \ref{Tor2}(ii). The same result holds, according to Dao \cite{Da1}, over certain hypersurface rings, such as hypersurfaces that are even dimensional simple singularities \cite[3.16]{Da1} or one-dimensional domains \cite[3.3]{Da1}. In fact, the following holds for one-dimensional hypersurfaces; see the beginning of Section 3.

\begin{prop} \label{ilginc} Let $R$ be a one-dimensional complete reduced hypersurface ring. Then $R$ is a domain if and only if each rigid $R$-module is free.
\end{prop}

Unlike the regular case, rigidity property of modules over several classes of rings remain mysterious. In this paper we are concerned with the following question, which is due to Dao \cite[9.1.4]{Lsurvey} for the Artinian case, and due to Huneke-Wiegand \cite[page 473]{HW1} for the one-dimensional case:

\begin{ques} \label{soru} Let $R$ be a Gorenstein local ring and let $M$ be an $R$-module. Assume either $R$ is Artinian, or $R$ is a one-dimensional domain and $M$ is torsion-free. If $M$ is rigid, then must $M$ be free? 
\end{ques}

Question \ref{soru} is intriguing to us as it is closely related to the celebrated conjecture of Auslander-Reiten \cite{AuRe} which asks whether maximal Cohen-Macaulay modules $M$ are free over Gorenstein rings $R$ provided that $\Ext^i_R(M,M)=0$ for all $i>0$. This conjecture, as well as its special cases such as Question \ref{soru}, are quite subtle to tackle, but they were intensively studied by several researchers via distinct arguments \cite{Ar, HJ, KK} since the conjecture intends to understand the general structure of modules, in fact originally over finite dimensional algebras. The one-dimensional case of Question \ref{soru}, which is also known as the Huneke-Wiegand conjecture \cite{HW1}, assumes $M\otimes_RM^{\ast}$ is torsion-free where $M^{\ast}=\Hom_R(M,R)$ \cite[5.9]{HJ}. There are one-dimensional Gorenstein domains $R$ and non-free torsion-free $R$-modules $M$ and $N$ (and also those that have both torsion) such that $M\otimes_R N$ is torsion-free \cite[2.1]{Constapel}. However, it is not known whether or not similar examples exist for tensor products of the form $M\otimes_RM^{\ast}$, even over one-dimensional complete intersection domains of codimension two.

The aim of this paper is to make progress on Question \ref{soru}, especially for periodic modules. Given an $R$-module $M$ and integer $n\geq 0$, we say $M$ is $n$-\emph{periodic} if $M \overset{st}{\cong} \Omega^n_R M$ where $\Omega^n_R M$ denotes the (minimal) $n$th syzygy module of $M$ and $\overset{st}{\cong}$  denotes a stable isomorphism, that is, $M \overset{st}{\cong} \Omega^n_R M$ if and only if $M \oplus F \cong \Omega^n_R M \oplus G$ for some free $R$-modules $F$ and $G$; see also \ref{stiso}.
 
Recall that $M$ has \emph{rank} (respectively, $M$ is \emph{generically free}) provided that there is an integer $r\geq 0$ such that $M_{\fp} \cong R_{\fp}^{\oplus r}$ (respectively, $M_{\fp}$ is a free $R_{\fp}$-module) for all $\fp \in \Ass(R)$. Our work is motivated by the following result which was proved independently by Le \cite[2.2.22 and 2.2.24]{Enni} and Toshinori Kobayashi (private communication):

\begin{thm} [{Le, Kobayashi}] \label{introFukano} If $R$ is a one-dimensional Gorenstein ring and $M$ is a $4$-periodic rigid $R$-module which has rank, then $M$ is free.
\end{thm}

Theorem \ref{introFukano}, for the case where $M$ is 2-periodic, was proved by Celikbas-Le-Matsui \cite[1.3]{Fukano}. If $R$ has dimension one and $M$ is an $R$-module which has rank, then $M$ is generically free and $[M]=0$ in $\overline{\G}(R)_{\QQ}$, where $\overline{\G}(R)_{\QQ}$ denotes the reduced Grothendieck group of $R$ with rational coefficients; see \ref{gg} and \ref{gglist}(ii) for the details. Motivated by this observation, we generalize Theorem \ref{introFukano} and replace the hypothesis ``$M$ has rank" with the weaker hypothesis ``$M$ is generically free and has zero class in the the reduced Grothendieck group." We obtain this generalization as a corollary of the following theorem.

\begin{thm}\label{yenithm} Let $R$ be a $d$-dimensional Gorenstein ring, $M$ be a generically free $R$-module, and let $n\geq 1$. Assume the following conditions hold:
\begin{enumerate}[\rm(i)]
\item $[M]=0$ in $\overline{\G}(R)_{\QQ}$.
\item $M$ is $2n$-periodic.
\item $\Ext^i_R(M,M)=0$ for all $i=1,3, \ldots, 2n-3$.
\end{enumerate}
Then $M$ is free if at least one the following holds:
\begin{enumerate}[\rm(1)]
\item $\Ext^{2n-1}_R(M,M)=0$.
\item $M\otimes_RM^{\ast}$ is torsion-free.
\item $\Ext^i_R(M,M)=0$ for all $i=1, 2,\ldots, \max\{1,d\}$.
\end{enumerate}
\end{thm}

We prove Theorem \ref{yenithm} as a corollary of two general theorems, namely Theorems \ref{genel} and \ref{mainthmnew}; see also Corollary \ref{maincor}. Note that Theorem \ref{yenithm}(3), when $n=2$, establishes Theorem \ref{introFukano}; see also Corollary \ref{corHiroki}. Although we consider Gorenstein rings in Theorem \ref{yenithm}, our work is not restricted to such rings. For example, in Theorem \ref{genel}, we consider $4$-periodic totally reflexive modules over Cohen-Macaulay rings which are not necessarily Gorenstein.
It is also worth noting here that there are examples of $2n$-periodic modules for $n\geq 2$; see, for example, \cite[3.3]{GP}.  

A consequence of our results pertaining to Question \ref{soru} for the Artinian case can be stated as follows; see Propositions \ref{thmintroproof1} and \ref{goodprop}, and also Corollary \ref{H4}.

\begin{prop} \label{thmintro} Let $M$ be a rigid $R$-module. Then $M$ is free if at least one of the following holds:
\begin{enumerate}[\rm(i)]
\item $R$ is Artinian and $M$ is $2$-periodic.
\item $R$ has depth zero and $M \overset{st}{\cong} \Omega^n_RI$ for some totally reflexive ideal $I$ of $R$ and integer $n\geq 0$.
\end{enumerate}
\end{prop}

We also give several applications of our results and investigate $2$-periodic rigid modules over rings of arbitrary dimension; see, Proposition \ref{2-periodic-prop}, Theorem \ref{mainthmnew}, and Corollary \ref{corGor}.

\section{Preliminaries} 

This section is devoted to various preliminary definitions and results necessary for the proofs presented in the subsequent sections. 

\begin{chunk} [{\textbf{Transpose}}] \label{tr} Let $M$ be an $R$-module and let $P_1\overset{f}{\rightarrow}P_0\rightarrow M\rightarrow 0$ be a minimal projective presentation (that is, part of a minimal free resolution of $M$). The \emph{transpose}  $\Tr_R M$ of $M$ is the cokernel of $f^{\ast}=\Hom_{R}(f,R)$, and hence is given by the following exact sequence; see \cite{AuBr}.
$$0\rightarrow M^*\rightarrow P_0^*\overset{f^*}{\rightarrow} P_1^*\rightarrow \Tr_R M\rightarrow 0.$$
Note that $\Tr_R M$ is uniquely determined (since so is a minimal free resolution of $M$). Moreover, we have:
\begin{equation}\tag{\ref{tr}.1}
M^{\ast}\overset{st}{\cong} \Omega_R^2 \Tr_R M \;\; \text{ and } \;\; \Tr_R \Tr_R M \overset{st}{\cong} M \;\;  \text{ so that } \;\; (\Tr_R M)^{\ast}  \overset{st}{\cong} \Omega^2_R M.
\end{equation}
\end{chunk}

\begin{chunk} [{\textbf{Cosyzygy}}] \label{c3} Let $M$ be an $R$-module and let $\{f_{1},f_{2},\dots, f_{s}\}$ be a minimal generating set of $M^{\ast}$. 
Let $\displaystyle{\delta: R^{\oplus s} \twoheadrightarrow M^{\ast}}$ be the morphism defined by $\delta(e_{i})=f_{i}$ for $i=1,2,\dots, s$, where $\{e_{1},e_{2},\dots, e_{s}\}$ is the standard basis for $R^{\oplus s}$. 
Then, composing the natural map $M \rightarrow M^{\ast\ast}$ with $\delta^{\ast}$, we obtain the exact sequence:
\begin{equation}\notag{}
M \stackrel{u}{\rightarrow} R^{\oplus s} \rightarrow M_{1} \rightarrow 0,
\end{equation}
where $u(x)=(f_{1}(x),f_{2}(x),\dots, f_{s}(x))$ for all $x\in N$. 
Then the morphism $u$ is a pushforward (or a left projective approximation) in the sense that the dual $u^*: R^{\oplus s} \to M^*$ is surjective.
Note that such a construction is unique, up to a non-canonical isomorphism; see, for example, \cite[page 62]{EG}. We call $M_1$ the {\it cosyzygy of $M$} and denote it by $\Omega^{-1}_R M$.

The following properties hold:
\begin{enumerate}[\rm(i)]
\item If $M$ is torsionless, that is, if the natural map $M \rightarrow M^{\ast\ast}$ is injective, then the map $u$ is also injective. So, if $M$ is torsionless, we have a short exact sequence of the form:
\begin{equation}\notag{}
0 \rightarrow M \stackrel{u}{\rightarrow} R^{\oplus s} \rightarrow \Omega^{-1}_R M \rightarrow 0.
\end{equation}
\item It follows by construction that $\Ext^1_R(\Omega^{-1}_R M,R)=0$. Thus $\Tr_R \Omega^{-1}_R M \overset{st}{\cong}  \Omega_R \Tr_R M$; see \cite[3.9]{AuBr}. 
Hence, by setting $\Omega^{-i}_R M =\Omega_R^{-1}\Big(\Omega_R^{-(i-1)}M\Big)$ for $i\geq 2$, we have that $\Omega^{-i}_R M \overset{st}{\cong}  \Tr \Omega^i_R \Tr_R M $ for all $i\geq 1$.
\item It follows that $ \Omega^{-i}_R \Tr_RM  \overset{st}{\cong} \Tr_R \Omega_R^i \Tr_R \Tr_R M \overset{st}{\cong} \Tr_R \Omega_R^i M$ for all $i\geq 0$.
\item If $M$ is totally reflexive, then $\Omega_R^m\Omega_R^nM \overset{st}{\cong} \Omega_R^{m+n}M$ for each integer $m$ and $n$.
\end{enumerate}
\end{chunk}


\begin{chunk} [{\textbf{Grothendieck Group}}]  \label{gg} The \emph{Grothendieck group} of the category of all $R$-modules is defined as
$$G(R)=\frac{\bigoplus_{[X] \in \Gamma}\ZZ\cdot [X]}{\big<[X]-[X']-[X''] \;|\; 0 \to X' \to X \to X'' \to 0 \textnormal{ is an exact sequence of $R$-modules} \big>}$$
where $\Gamma$ is a set of isomorphism classes of all $R$-modules.

We denote the \emph{reduced Grothendieck group} by $\displaystyle{\overline{G}(R)=G(R)/\ZZ \cdot [R]}$ and the \emph{reduced Grothendieck group of $R$ with rational coefficients} by $\displaystyle{\overline{G}(R)_{\QQ}=(G(R)/\ZZ \cdot [R])\otimes_{\ZZ}\QQ}$; see, for example, \cite{Da1}. Given an $R$-module $M$, we also denote the class of $M$ in $\overline{G}(R)_{\QQ}$ by $[M]$. 
\end{chunk}

\begin{chunk} \label{gglist} We need the following properties in the sequel:
\begin{enumerate}[\rm(i)]
\item If $N$ is an $R$-module, then $[N]=0$ in $\overline{\G}(R)_{\QQ}$ provided that $N$ is a syzygy of a finite length $R$-module. Therefore, if $R$ is Artinian, then $\overline{\G}(R)_{\QQ}=0$; see \cite[2.5]{CeD}.
\item Assume $R$ has dimension one and $M$ is an $R$-module which has rank. Then the torsion-free part $\overline{M}$ of $M$ has rank and so there is an exact sequence of $R$-modules $0 \to \overline{M} \to F \to C \to 0$, where $F$ is free and $C$ is a module of finite length. As $[F]$ and $[C]$ are both zero in $\overline{\G}(R)_{\QQ}$, so is $[\overline{M}]$; see part (i). It follows from the exact sequence $0\to T \to M \to \overline{M} \to 0$ that $[M]=0$ in $\overline{\G}(R)_{\QQ}$ because the torsion submodule $T$ of $M$ has finite length. This argument shows that, if $R$ is a one-dimensional domain, then $\overline{\G}(R)_{\QQ}=0$.
\item If $R$ is a two-dimensional normal domain with torsion class group, then $\overline{\G}(R)_{\QQ}=0$; see \cite[2.5]{CeD}.
\item If $R$ is Gorenstein and $M$ is an $R$-module such that $M_{\fp}$ is a maximal Cohen-Macaulay $R_{\fp}$-module for all $\fp \in \Spec(R)-\{\fm\}$, then $[M]=0$ in $\overline{\G}(R)_{\QQ}$ if and only if $[M^{\ast}]=0$ in $\overline{\G}(R)_{\QQ}$ if and only if $[\Tr_R M] = 0$ in $\overline{\G}(R)_{\QQ}$;  see \cite[2.17]{Fukano}.
\end{enumerate}
\end{chunk}

\begin{chunk} \label{extiso} A reflexive $R$-module $M$ is called \emph{totally reflexive} if $\Ext^i_R(M,R)=\Ext^i_R(M^{\ast},R)=0$ for all $i>0$; see \cite{AuBr}.  

Let $M$ and $N$ be totally reflexive $R$-modules. Given $n\geq 1$, it follows that: $$\Ext^n_R(M,N) \cong  \Ext^n_R(N^{\ast}, M^{\ast}) \cong \Ext^n_R(\Tr_R N, \Tr_R M).$$
Moreover, for all  $n\geq 1$ and all $j\in \ZZ$ such that $n>j$, we have:
$$\Ext^{n-j}_R(M,N) \cong \Ext^n_R(M, \Omega^j_R N).$$ 
\end{chunk}

\begin{chunk} An $R$-module $M$ is called a \emph{vector bundle} if $M_{\fp}$ is a free $R_{\fp}$-module for all $\fp \in \Spec(R)-\{\fm\}$.
\end{chunk}

\begin{chunk} \label{Tor2} Let $M$ be an $R$-module. 
\begin{enumerate}[\rm(i)]
\item If $\Tor_1^R(M, \Tr_R M)=0$, then $M$ is free; see \cite[3.9]{Yo}.
\item If $M$ is rigid (that is, $\Ext^1_R(M,M)=0$) and the pair $(M, \Tr_R \Omega_R M)$ is Tor-rigid, then $M$ is free. Similarly, if $\Ext^n_R(M,M)=0$ for some $n\geq 0$ and $M$ is Tor-rigid, then $\pd_R(M)<n$; see, for example, \cite[3.1.2]{Lsurvey}
\item In view of \cite[2.6]{AuBr} there is an exact sequence of $R$-modules $$0 \to \Tor_2^R(M, \Tr_R M) \to M\otimes_RM^{\ast} \to \Hom_R(M,M) \to \Tor_1^R(M, \Tr_R M) \to 0.$$
Assume $\depth_R(M\otimes_R M^{\ast})\geq 1$ and $\Tor_2^R(M, \Tr_R M)$ has finite length. Then $\Tor_2^R(M, \Tr_R M)=0$ due to the injection $\Tor_2^R(\Tr_R M, M) \hookrightarrow M\otimes_RM^{\ast}$.

Recall that $(\Tr_R M)^{\ast} \overset{st}{\cong} \Omega_R^2 M$; see (\ref{tr}.1). Hence, as before, if $\depth_R(\Tr_R M\otimes_R \Omega^2_R M)\geq 1$ and $\Tor_2^R(M, \Tr_R M)$ has finite length, it follows that $\Tor_2^R(M, \Tr_R M)=0$.
\item Assume $\depth_R(M)\geq 1$. If $\Ext^i_R(\Tr_R M, R)$ has finite length for some $i\geq 1$, for example, if $M$ is a vector bundle, then $\Hom_R(\Ext^i_R(\Tr_R M, R),M)=0$.
\item It follows from \cite{AuBr} that, for each $i\geq 0$,  there is an exact sequence of the form $$0 \to \Ext^1_R(\Omega^{-i}_R M, M) \to \Tor_i^R(\Tr_R M, M) \to \Hom_R(\Ext^i_R(\Tr_R M, R), M).$$
\item Let $R$ be a $d$-dimensional Cohen-Macaulay ring with canonical module $\omega_R$, and let $M$ and $N$ be $R$-modules. Assume $M$ is a vector bundle and $N$ is maximal Cohen-Macaulay. Then $\Ext^i_R(M,N)=0$ for all $i=1, \ldots, d$ if and only if $M\otimes_RN^{\dagger}$ is maximal Cohen-Macaulay; see \cite[2.3]{GTAR}, and also \cite[5.9]{HJ} for the Gorenstein case.
\end{enumerate}
\end{chunk}

Since we are primarily concerned with periodic modules, the following observation is worth noting.

\begin{chunk} \label{stiso} Let $M$ be an $R$-module such that $M$ does not have nonzero free summands. Given an integer $n\geq 1$, it follows that $M$ is $n$-periodic if and only if $M \cong \Omega^n_R M$.

To see this, assume $M$ is $n$-periodic. Then $M\oplus R^{\oplus a} \cong \Omega^n_R M$ for some $a\geq 0$; see \cite[1.10, 1.15]{TheBook}. By the depth lemma, we have that $\depth_R(\Omega^j_R M)\geq \inf\big\{\depth(R), \depth_R(M)+j\big\}$ for all $j\geq 0$. Hence, if $\depth_R(M)<\depth(R)$, then $\depth_R(M)=\inf\{\depth_R(M), \depth(R)\}=\depth_R(M\oplus R^{\oplus a})=\depth_R(\Omega^n_R M)$. This shows that $\depth_R(M)\geq \inf\big\{\depth(R), \depth_R(M)+n\big\}$, a contradiction with the assumptions $n\geq 1$ and $\depth_R(M)<\depth(R)$. Thus, we conclude that $\depth_R(M)\geq \depth(R)$.

As $M\oplus R^{\oplus a} \cong \Omega^n_R M$, there is a short exact sequence of $R$-modules $0 \to M\oplus R^{\oplus a} \to R^{\oplus b} \to \Omega^{n-1}_R M\to 0$ where $b=\mu_R(\Omega^{n-1}_R M)$, the minimal number of generators of $\Omega^{n-1}_R M$. This sequence yields two exact sequences $0\to  R^{\oplus a} \to  R^{\oplus b} \to B \to 0$ and $0\to M \to  B \to \Omega^{n-1}_R M \to 0$ for some $R$-module $B$; see, for example, \cite[3.1]{TakGS}. We see, by using the depth lemma and the latter sequence, that $\depth_R(B)\geq \depth_R(M)$ since $\depth_R(M)\geq \depth(R)$.
As the former sequence shows that $\pd_R(B)<\infty$, we deduce that $B$ is free and $R^{\oplus b} \cong R^{\oplus a} \oplus B$. Therefore, $\mu_R(\Omega^{n-1}_R M)=b=a+\mu_R(B)\geq a+\mu_R(\Omega^{n-1}_R M)$. So, $a=0$ and $M \cong \Omega^n_R M$.
\end{chunk}

\section{Proofs of Propositions  \ref{ilginc} and \ref{thmintro}} 

We start this section by giving a proof for Proposition \ref{ilginc} stated in the introduction:

\begin{proof}[Proof of Proposition \ref{ilginc}] Assume $R$ is a one-dimensional hypersurface domain. Then each $R$-module is Tor-rigid \cite[3.3]{Da1} so that \ref{Tor2}(ii) shows that rigid modules are free. Hence it suffices to prove the converse implication.

Let $R=S/(f)$ for some complete two-dimensional regular local ring $(S, \fn)$ and an element $0\neq f\in \fn$. We can write $f=p_1\cdots p_n$ for some $n\geq 1$, where each $p_i$ are prime elements in $S$. Furthermore, $(p_i)$ and  $(p_j)$ are distinct ideals of $S$ for distinct $i$ and $j$ since $R$ is reduced.

Suppose $R$ is not a domain. Then $n>1$. Set $p=p_1$ and $q=p_2\cdots p_n$ so that $R=S/(pq)$. Then the minimal free resolution of the $R$-module $M=R/(p)$ is $\cdots \to R \xrightarrow{p} R\xrightarrow{q} R \xrightarrow{p} R \to 0$. It follows that $M$ is rigid since $\Ext^1_R(M,M) \cong \Ker(M \xrightarrow{q} M)=0$. As $M$ is not free, this provides a contradiction of our assumption. Consequently, $n=1$ and $R$ is a domain.
\end{proof}

Next we proceed to prove Proposition \ref{thmintro}. In fact, the first part of the proposition is a quick application of the following result:

\begin{chunk} [{\cite[3.5]{Fukano}}] \label{teta} Let $M$ and $N$ be $R$-modules. Assume the following conditions hold:
\begin{enumerate}[\rm(i)]
\item $[N]=0$ in $\overline{\G}(R)_{\QQ}$.
\item $\pd_{R_{\fp}}(M_{\fp})<\infty$ for all $\fp \in \Spec(R)-\{\fm\}$.
\item $M$ is $2$-periodic for some $n\geq 1$.
\end{enumerate}
Then $(M,N)$ is Tor-rigid, that is, if $\Tor_n^R(M,N)=0$ for some $n\geq 0$, then $\Tor_i^R(M,N)=0$ for all $i\geq n$.
\end{chunk}

The following establishes Theorem \ref{thmintro}(i):

\begin{prop} \label{thmintroproof1} Let $R$ be an Artinian ring and let $M$ be a $2$-periodic $R$-module. If $\Ext^j_R(M,M)=0$ for some $j\geq 0$, then $M$ is free.
\end{prop}

\begin{proof} As $R$ is Artinian and $M$ is a $2$-periodic $R$-module, it follows that $\overline{\G}(R)_{\QQ}=0$ and so $M$ is Tor-rigid; see \ref{gglist}(i) and \ref{teta}. Therefore, due to \ref{Tor2}(ii), if $\Ext^j_R(M,M)=0$ for some $j\geq 0$, then $M$ is free.
\end{proof}

As a corollary of the next result, we prove the second part of Proposition \ref{thmintro}; see Corollary \ref{H4}.

\begin{lem} \label{goodprop} Let $I$ and $J$ be ideals of $R$ such that $I\subseteq J$. Assume $\Ext_R^1(R/I,R)=\Ext_R^1(I,J)=0$. Then:
\begin{enumerate}[\rm(i)]
\item $\Hom_R(I,R/J)=0$. 
\item If $\depth_R(R/J)=0$, then $I=0$.
\end{enumerate}
\end{lem}

\begin{proof} 
(i) As $\Ext_R^1(I,J)=0$, the natural exact sequence $0\to J\xrightarrow{}R\to R/J\to0$ induces the exact sequence $0\to\Hom_R(I,J)\xrightarrow{\xi}\Hom_R(I,R)\to\Hom_R(I,R/J)\to0$.
Similarly, since $\Ext_R^1(R/I,R)=0$, it follows from the natural exact sequence $0\to I\xrightarrow{\alpha}R\to R/I\to0$ that $\alpha^{\ast}:R^{\ast} \to I^{\ast}$ is surjective.

Let $f\in I^{\ast}$. As $\alpha^{\ast}$ is surjective, $f$ factors through $\alpha$, that is, there exists $c\in R$ such that $f=c \cdot \alpha$. Now define a map $g:I\to R$ given by $g(x)=c \cdot x$ for all $x\in I$. Then $g \in \Hom_R(I,J)$ since $I\subseteq J$.
Therefore $f=\xi(g)$ and so $\xi$ is surjective. This implies that $\Hom_R(I,R/J)=0$. 

(ii) We know by part (i) that $\Hom_R(I,R/J)=0$. So \cite[1.2.3]{BH} implies that there exists a non zero-divisor on $R/J$ which is contained in $\ann_R(I)$. As $\depth_R(R/J)=0$, this yields a contradiction if $\ann_R(I)\subseteq \fm$. So $\ann_R(I)=R$, that is, $I=0$.
\end{proof}

\begin{cor}\label{H4} Let $R$ be a ring such that $\depth(R)=0$ and let $I$ and $J$ be ideals of $R$. Assume $I$ and $J$ are totally reflexive as $R$-modules and $I\subseteq J$. If $\Ext_R^1(\Omega_R^n I, \Omega_R^n J)=0$ for some $n\in\ZZ$, then $I=0$.
\end{cor}

\begin{proof} Note that $\Ext_R^1(R/I,R)=0$ and $\depth_R(R/J)=\depth(R)$ since $R/I$ and $R/J$ are totally reflexive $R$-modules. Moreover, it follows that $\Ext_R^1(I,J)\cong \Ext_R^1(\Omega_R^n I, \Omega_R^n J)=0$; see \ref{extiso}. Hence, by Lemma \ref{goodprop}(ii), we have that $I=0$.
\end{proof}

In passing let us note that Corollary \ref{H4} recovers the following result of Lindo; see \cite[3.9]{Lindo2}:

\begin{cor} Let $R$ be an Artinian Gorenstein ring and let $M \overset{st}{\cong} \Omega_R^n I $ for some $n\in \ZZ$. If $\Ext_R^1(M,M)=0$, then $M$ is free.
\end{cor}

\section{Proof of part (3) of Theorem \ref{yenithm}}

In this section we give a proof of part (3) of Theorem \ref{yenithm}; the result follows as an immediate corollary of Theorem \ref{genel}. First we need to make several preparations.

\begin{prop} \label{2-periodic-prop} Let $M$ and $N$ be $R$-modules. Assume:
\begin{enumerate}[\rm(i)]
\item $[M^{\ast}]=0$ in $\overline{\G}(R)_{\QQ}$.
\item $M$ is generically free. 
\item $M$ is $2$-periodic and rigid.
\end{enumerate}
Then $M$ is free.
\end{prop}

\begin{proof} We may assume, by induction on the dimension of $R$, that $\dim(R)>0$ and $M$ is a vector bundle. Hence, by \ref{Tor2}(ii) and  \ref{teta}, it is enough to see that $[\Tr_R \Omega_R M]=0$ in $\overline{\G}(R)_{\QQ}$, or equivalently $[(\Omega_R M)^{\ast}]=0$ in $\overline{\G}(R)_{\QQ}$. Note that $\Ext^1_R(M,R)$ has finite length since $M$ is a vector bundle. Therefore, $[\Ext^1_R(M,R)]=0$ in $\overline{\G}(R)_{\QQ}$. So, by dualizing the syzygy exact sequence $0 \to \Omega_R M \to F \to M \to 0$, we see that $[(\Omega_R M)^{\ast}]=0$ in $\overline{\G}(R)_{\QQ}$ if and only if $[M^{\ast}]=0$ in $\overline{\G}(R)_{\QQ}$. Thus $[\Tr_R \Omega_R M]=0$ in $\overline{\G}(R)_{\QQ}$ as we assume $[M^{\ast}]=0$.
\end{proof}

\begin{cor}  \label{2-periodic-prop-cor} Let $R$ be a Gorenstein ring and let $M$ and $N$ be $R$-modules. Assume:
\begin{enumerate}[\rm(i)]
\item $[M]=0$ in $\overline{\G}(R)_{\QQ}$.
\item $M$ is generically free. 
\item $M$ is $2$-periodic and rigid.
\end{enumerate}
Then $M$ is free.
\end{cor}

\begin{proof} Note that $M$ is maximal Cohen-Macaulay since $M$ is $2$-periodic. Therefore, $[M^{\ast}]=0$ in $\overline{\G}(R)_{\QQ}$; see \ref{gglist}(iv). Thus $M$ is free by Proposition \ref{2-periodic-prop}.
\end{proof}

\begin{chunk} \label{Yas} Let $R$ be a Cohen-Macaulay ring with a canonical module $\omega_R$ and let $M$ be a totally reflexive $R$-module. Then $\Tor_i^R(M, \omega_R) = 0$ for all $i>0$ and $M\otimes_R\omega_R$ is maximal Cohen-Macaulay; see, for example, \cite[2.11]{Yass}.
\end{chunk}

In the following $(-)^\dagger$ denotes $\Hom_R(-, \omega_R)$.

\begin{lem}\label{totref} Let $R$ be a Cohen-Macaulay ring with a canonical module $\omega_R$, and let $M$ and $N$ be totally reflexive $R$-modules. Then the following hold:
\begin{enumerate}[\rm(i)]
\item $(\Tr_RM)^\dagger \cong \Omega_R^2 M \otimes_R \omega_R$.
\item	$\Ext_R^i(M, N) \cong \Ext_R^i(M \otimes_R \omega_R, N\otimes_R \omega_R)$ for all $i \ge 0$.
\end{enumerate}
\end{lem}

\begin{proof} In the proof we make use of the fact that $\Tor_i^R(M, \omega_R) = 0$ for all $i>0$; see \ref{Yas}.

(i) Take a minimal presentation of $M$, say $F_1 \xrightarrow{f} F_0 \to M \to 0$. By dualizing this presentation, we get the following exact sequence:
\begin{align}\tag{\ref{totref}.1}
\begin{aligned}
0 \to M^* \to F_0^* \xrightarrow{f^*} F_1^* \to \Tr_R M \to 0	
\end{aligned}
\end{align}
Then, by applying $(-)^\dagger$ to (\ref{totref}.1), we get the exact sequence
\begin{align}\tag{\ref{totref}.2}
\begin{aligned}
0 \to (\Tr_R M)^\dagger \to F_1^{*\dagger} \xrightarrow{f^{*\dagger}} F_0^{*\dagger} \to M^{*\dagger} \to 0
\end{aligned}
\end{align}
since $\Tr_R M$ is totally reflexive, and hence is maximal Cohen-Macaulay.

Recall that $\Tor_i^R(M, \omega_R) = 0$ for all $i>0$. Hence, by applying $-\otimes_R\omega_R$ to the short exact sequence $0 \to \Omega_R^2 M \to F_1 \xrightarrow{f} F_0 \to M \to 0$, we obtain the exact sequence:
\begin{align}\tag{\ref{totref}.3}
\begin{aligned}
0 \to \Omega_R^2 M \otimes_R \omega_R \to F_1 \otimes_R \omega_R \xrightarrow{f \otimes 1} F_0 \otimes_R \omega_R \to M \otimes_R \omega_R \to 0.
\end{aligned}
\end{align}

Note that there is a commutative diagram as:
\[
\xymatrix{
& F_1^{*\dagger}  \ar[r]^{f^{*\dagger}} \ar[d]_{\cong} & F_0^{*\dagger}  \ar[d]_{\cong} &   & \\
 & F_1 \otimes_R \omega_R  \ar[r]^{f \otimes 1}  & F_0 \otimes_R \omega_R   &  & \\
}
\]
Therefore, in view of (\ref{totref}.2) and (\ref{totref}.3), we conclude that:
$$
(\Tr_R M)^\dagger \cong \Ker(F_1^{*\dagger} \xrightarrow{f^{*\dagger}} F_0^{*\dagger}) \cong \Ker(F_1 \otimes_R \omega_R \xrightarrow{f \otimes 1} F_0 \otimes_R \omega_R) \cong \Omega_R^2 M \otimes_R \omega_R.
$$ 

(ii) As $\Tor_i^R(M, \omega_R) = 0$ for all $i>0$, the following isomorphisms hold in the derived category of $R$-modules: 
\begin{align*}
\RHom_R(M \otimes_R \omega_R, N\otimes_R \omega_R) &\simeq \RHom_R(M \ltensor_R \omega_R, N \ltensor_R \omega_R) \\
&\simeq \RHom_R\big(M , \RHom_R(\omega_R, N \ltensor_R \omega_R)\big) \\
&\simeq \RHom_R(M, N),	 	 
\end{align*}
where the second and the third isomorphisms follow from \cite[A.4.21]{Gdimbook} and \cite[4.2.6]{Gdimbook}, respectively. This establishes the required isomorphism of $\Ext$ modules.
\end{proof}

\begin{chunk} \label{KYT} Let $R$ be a Cohen-Macaulay ring with a canonical module $\omega_R$ and let $X$ and $Y$ be $R$-modules. If $Y$ is maximal Cohen-Macaulay, then $X\otimes_RY$ is maximal Cohen-Macaulay if and only if $Y^{\dagger}\otimes_R \Tr_R X$ is maximal Cohen-Macaulay; see \cite[2.2]{KYT}.
\end{chunk}

\begin{lem} \label{conf} Let $R$ be a $d$-dimensional Cohen-Macaulay ring with a canonical module $\omega_R$, and let $M$ and $N$ be $R$-modules. Assume the following hold:
\begin{enumerate}[\rm(i)]
\item $M$ and $N$ are vector bundles.
\item $M$ is totally reflexive.
\item $N$ is maximal Cohen-Macaulay
\end{enumerate}
Then $\Ext_R^i(N, \Omega^2_RM \otimes_R \omega_R)=0$ for all $i=1,2, \ldots, d$ if and only if $\Ext_R^i(M, N)=0$ for all $i=1,2, \ldots, d$.
\end{lem}

\begin{proof} Note that $\Tr_R M$ is totally reflexive so that $(\Tr_R M)^{\dagger}$ is maximal Cohen-Macaulay. We have the following implications:
\begin{align}\notag{}
\Ext^i_R(M,N)=0 \mbox{ for all $i=1, \ldots,d$}& \Longleftrightarrow  M\otimes_R N^{\dagger} \text{ is maximal Cohen-Macaulay}  \\ & \notag{} \Longleftrightarrow N \otimes_R \Tr_R M  \text{ is maximal Cohen-Macaulay } \\ & \notag{} \Longleftrightarrow   \Ext^i_R\big(N,(\Tr_R M)^{\dagger}\big) =0 \mbox{ for all $i=1,2, \ldots,d$} \\ & \notag{}  \Longleftrightarrow  \Ext^i_R(N,\Omega^2_RM \otimes_R \omega_R )=0 \mbox{ for all $i=1,2,\ldots,d$}
\end{align}	
Here the first implication holds by \ref{Tor2}(vi) since $M$ is a vector bundle and $N$ is maximal Cohen-Macaulay; the second implication holds by \ref{KYT} since $N^{\dagger}$ is maximal Cohen-Macaulay; the third implication holds by \ref{Tor2}(vi) since $N$ is a vector bundle and $(\Tr_R M)^{\dagger}$ is maximal Cohen-Macaulay, and the fourth implication holds by Lemma \ref{totref}(i) since $M$ is totally reflexive.
\end{proof}

Next is the main result of this section; its Gorenstein case establishes part (3) of Theorem \ref{yenithm}.

\begin{thm} \label{genel} Let $R$ be a $d$-dimensional Cohen-Macaulay local ring with a canonical module $\omega_R$, and let $M$ be an $R$-module. Assume:
\begin{enumerate}[\rm(i)]
\item $R_{\fp}$ is Gorenstein for all $\fp \in \Spec(R)-\{\fm\}$.
\item $M$ is $2n$-periodic and $\Ext^i_R(M,M)=0$ for all $i=1,3, \ldots, 2n-3$ and for some integer $n\geq 2$. 
\item $M$ is generically free, totally reflexive, and $[M^{\ast}]=0$ in $\overline{\G}(R)_{\QQ}$.
\item $\Ext_R^i(M,M \otimes_R \omega_R)=0$ for all $i=1,2, \ldots, d$.
\end{enumerate}
Then $M$ is free.	
\end{thm}

\begin{proof} As $M$ is generically free, we may assume, by induction on the dimension of $R$, that $R$ has positive dimension and $M$ is a vector bundle. 

Letting $N=M \otimes_R \omega_R$, in view of \ref{Yas}, we conclude from Lemma \ref{conf} that $\Ext_R^i(M \otimes_R \omega_R, \Omega^2_R M \otimes_R \omega_R)=0$  for all $i=1, \ldots, d$. Now, Lemma \ref{totref}(ii) yields $\Ext_R^i(M, \Omega^2_RM) = 0$ for all $i=1,2, \ldots, d$.  As $M$ is $2n$-periodic, it follows that  $\Ext_R^i(M, \Omega^2_R M)=0$ for all $i=2n+1,2n+2, \ldots, 2n+d$. Note that, since $M$ is totally reflexive, we have $\Ext^i_R(M,R)=0$ for all $i>0$. Thus, we conclude that  $\Ext_R^i(M, M)=0$ for all $i=2n-1,2n, \ldots, 2n+d-2$. As we assume $\Ext^i_R(M,M)=0$ for all $i=1, 3, \ldots, 2n-3$, it follows that $\Ext_R^i(M, M)=0$ for each odd integer $i$ such that $1\leq i \leq 2n+d-2$. 

Set $X=M\oplus \Omega^2_R M \oplus \cdots \oplus \Omega^{2n-2}_R M$. Then $X$ is generically free, rigid, $2$-periodic, and $[X^{\ast}]=0$ in $\overline{\G}(R)_{\QQ}$. Thus, by Proposition \ref{2-periodic-prop}, $X$ is free. Therefore, $M$ is free.
\end{proof}

Next we give a corollary of Theorem \ref{genel}; the case where $d=1$ and $n=2$ of the corollary yields a proof of Theorem \ref{introFukano}.

\begin{cor} \label{corHiroki} Let $R$ be a $d$-dimensional Gorenstein local ring and let $M$ be an $R$-module. Assume:
\begin{enumerate}[\rm(i)]
\item $M$ is $2n$-periodic for some integer $n$ such that $1\leq n \leq (d+3)/2$. 
\item $M$ is generically free and $[M]=0$ in $\overline{\G}(R)_{\QQ}$.
\item $\Ext_R^i(M,M)=0$ for all $i=1,2, \ldots, d$.
\end{enumerate}
Then $M$ is free.
\end{cor}

\begin{proof} It follows from our assumption that $2n-3 \leq d$. Hence the claim follows from Theorem \ref{genel}.
\end{proof}

\section{Proof of parts (1) and (2) of Theorem \ref{yenithm}}

In this section we prove parts (1) and (2) of Theorem \ref{yenithm}; this is a direct consequence of Corollary \ref{maincor}. We also make an observation in Proposition \ref{corGor} on periodic modules that is related to Question \ref{soru}.

The following theorem is the main result of this section:

\begin{thm} \label{mainthmnew} Let $M$ be an $R$-module. Assume:
\begin{enumerate}[\rm(i)]
\item $[M^{\ast}]=0$ in $\overline{\G}(R)_{\QQ}$.
\item $M$ is generically free.
\item $\Tr_R M$ is $2n$-periodic for some $n\geq 1$.  
\item $\Ext^1_R(\Omega^{-i}_R M, M)=0$ for all $i=4, 6, \ldots, 2n$. 
\end{enumerate}
Then $M$ is free if either (1) or (2) holds:
\begin{enumerate}[\rm(1)]
\item
\begin{enumerate}[\rm(a)]
\item $M$ is a vector bundle.
\item Either $\depth_R(M)\geq 1$ and $\Ext^1_R(\Omega^{-2}_R M, M)=0$, or $\depth_R(M\otimes_RM^{\ast})\geq 1$ and $\Ext^i_R(\Tr_R M, R)=0$ for all $i=4, 6, \ldots, 2n$.
\end{enumerate}
\item 
\begin{enumerate}[\rm(a)]
\item $R$ has no embedded primes.
\item $\Omega_R^2M \otimes_R \Tr_R M$ is torsion-free.
\item $\Ext^1_R(\Tr_R M,R)=0$ or $\Ext^i_R(\Tr_R M,R)=0$ for all $i=4, 6, \ldots, 2n$.
\end{enumerate}
\end{enumerate}
\end{thm}

We defer the proof of Theorem \ref{mainthmnew} and next discuss the following consequence of it.

\begin{cor} \label{maincor} Let $M$ be a totally reflexive $R$-module. Assume:
\begin{enumerate}[\rm(i)]
\item $[M]=0$ in $\overline{\G}(R)_{\QQ}$.
\item $M$ is generically free.
\item $M$ is $2n$-periodic for some $n\geq 1$.
\item $\Ext^i_R(M,M)=0$ for all $i=1,3, \ldots, 2n-3$.
\end{enumerate}
Then $M$ is free if at least one the following holds:
\begin{enumerate}[\rm(1)]
\item $M$ is a vector bundle and $\Ext^{2n-1}_R(M,M)=0$.
\item $R$ has no embedded primes and $M\otimes_RM^{\ast}$ is torsion-free.
\end{enumerate}
\end{cor}

\begin{proof}  Set $N=\Tr_R M$. Then it suffices to show that $N$ is free; see (\ref{tr}.1). Also, by our assumptions, we have that $N$ is generically free, $\Tr_R N$ is $2n$-periodic, and $[N^{\ast}]=0$ in $\overline{\G}(R)_{\QQ}$; see (\ref{tr}.1).

Given an integer $i\geq 1$, we have the following isomorphisms:
\begin{align}\notag{}
\Ext^1_R(\Omega^{-i}_R N, N) & \cong \Ext^1_R(\Omega_R^{-i}\Tr_R M, \Tr_R M) \\ & \cong \Ext^1_R(\Tr_R \Omega_R^{i}M, \Tr_R M) \notag{} \\ & \cong \Ext^1_R(M, \Omega_R^i M) \notag{} \\ & \cong \Ext^1_R(\Omega^{2n}_R M, \Omega_R^i M) \tag{\ref{maincor}.1} \\ & \cong \Ext^{2n+1}_R(M, \Omega_R^i M) \notag{} \\ & \cong \Ext^{2n-i+1}_R(M, M) \notag{} 
\end{align}
Here, in (\ref{maincor}.1), the second isomorphism follows from \ref{c3}(iii), namely since $\Omega^{-i}_R \Tr_R M \cong \Tr_R \Omega_R^i M$, the third and the sixth isomorphisms follow from \ref{extiso}, and the forth isomorphism holds since $M$ is $2n$-periodic. Hence, $\Ext^1_R(\Omega^{-i}_R N, N)=0$ for all $i=4, 6, \ldots, 2n$ since $\Ext^i_R(M,M)$ for all $i=1,3, \ldots, 2n-3$. Consequently, $N$ satisfies conditions (i), (ii), (iii), and (iv) of Theorem \ref{mainthmnew}.

Now assume (1) holds. Then $N$ is a vector bundle since so is $M$. As $M$ is generically free, we may assume $\depth(R)\geq 1$. Thus, $\depth_R(N)=\depth(R)\geq 1$ because $N$ is totally reflexive.  As we assume $\Ext^{2n-1}_R(M, M)=0$, (\ref{maincor}.1) implies that $\Ext^{1}_R(\Omega^{-2}_R N, N)=0$. So, $N$ is free due to part (1) of Theorem \ref{mainthmnew}.

Next assume (2) holds. Then $\Omega_R^2N \otimes_R \Tr_R N$ is torsion-free since it is isomorphic to a finite direct sum of copies of $M^{\ast} \otimes_RM$, $M$, $M^{\ast}$ and free $R$-modules; see (\ref{tr}.1). As $\Tr_R N \overset{st}{\cong} M$ is totally reflexive, we have that $\Ext^i_R(\Tr_R N, R)=0$ for all $i\geq 1$. 
Hence, $N$ is free due to part (2) of Theorem \ref{mainthmnew}.
\end{proof}

\begin{proof}[Proof of Parts (1) and (2) of Theorem \ref{yenithm}] First assume $\Ext^{2n-1}_R(M,M)=0$. We may assume, by induction on $d$, that $M$ is a vector bundle. Hence, $M$ is free by Corollary \ref{maincor}(1).

Next assume $M\otimes_RM^{\ast}$ is torsion-free. As $R$ is Gorenstein, it has no embedded primes so that $M$ is free by Corollary \ref{maincor}(2).
\end{proof}

The hypothesis that $M$ is generically free, or $[M]=0$ in $\overline G(R)_{\mathbb Q}$, is necessary for Corollary \ref{maincor}. We highlight this fact with the next examples:

\begin{eg} \label{ex1} $\phantom{}$
\begin{enumerate}[\rm(i)]
\item Let $R=k[\![x,y]\!]/(x^2)$ and $M=R/(x)$. Then $M$ is nonfree, $2$-periodic, rigid, and $[M]=0$ in $\overline{\G}(R)_{\QQ}$. Moreover, $\fp=(x) \in \Ass(R)$ and $M_{\fp}$ is not free over $R_{\fp}$. Hence, $M$ is not generically free.
\item Let $R=k[\![x,y]\!]/(xy)$ and $M=R/(x)$. Then $M$ is nonfree, $2$-periodic, and rigid. Moreover, $M$ is generically free since $R$ is reduced. On the other hand, $[M]\neq 0$ in $\overline{\G}(R)_{\QQ}$: One way to see this is to set $N=R/(y)$, note $\Tor_1^R(M,N)=0\neq \Tor_2^R(M,N)$, and use \cite[2.8 and 3.1]{Da1} (in fact it follows that $\G(R)=\ZZ\cdot [M] \oplus \ZZ\cdot [N]$ and $\overline{\G}(R)=\ZZ\cdot [M]$).
\end{enumerate}
\end{eg}

Next we prepare a lemma and subsequently use it to prove Theorem \ref{mainthmnew}.

\begin{lem} \label{mainlemma} Let $M$ be an $R$-module and let $n\geq 1$. Assume:
\begin{enumerate}[\rm(i)]
\item $[M^{\ast}]=0$ in $\overline{\G}(R)_{\QQ}$.
\item $M$ is generically free. .
\item $\Tr_R M$ is $2n$-periodic.
\item $\Tor_i^R(M, \Tr_R M)=0$ for all $i=2, 4, \ldots, 2n$.
\end{enumerate}
Then $M$ is free.
\end{lem}

\begin{proof} We may assume, by induction on the dimension of $R$, that $R$ has positive dimension and $M_{\fp}$ is free for all $\fp \in \Spec(R)-\{\fm\}$. Set $X=\Tr_R M \oplus \Omega^2 \Tr_R M \oplus \Omega^4 \Tr_R M \oplus \cdots \oplus \Omega^{2n-2}\Tr_R M$. Then $X$ is $2$-periodic since $\Tr_R M$ is $2n$-periodic. Note also that $\pd_{R_{\fp}}(X_{\fp})<\infty$ for all $\fp \in \Spec(R)-\{\fm\}$ and $[X]=0$ in $\overline{\G}(R)_{\QQ}$. Hence, it follows from \ref{teta} that the pair $(X,M)$ is Tor-rigid. 

Suppose $\Tor_i^R(\Tr_R M, M) =0$ for all $i=2,4, \ldots, 2n$. Then $\Tor_2^R(X,M)=0$, and so $\Tor_3^R(X,M)=0$ since $(X,M)$ is Tor-rigid. Note that $\Tor_1^R(M, \Tr_R M) \cong \Tor_1^R(\Omega^{2n}_R \Tr_R M, M) \cong \Tor_3^R(\Omega^{2n-2}\Tr_R M,M)$, and $\Tor_3^R(\Omega^{2n-2}\Tr_R M,M)=0$ as it is a direct summand of $\Tor_3^R(X,M)$. Consequently $\Tor_1^R(M, \Tr_R M)=0$ and so $M$ is free; see \ref{Tor2}(i).
\end{proof}

Next is the proof of the main result of this section:

\begin{proof}[Proof of Theorem \ref{mainthmnew}] $\phantom{}$

(1) In view of Lemma \ref{mainlemma}, it suffices to prove $\Tor_i^R(M, \Tr_R M)=0$ for all $i=2, 4,\ldots, 2n$. Assume first $\depth_R(M)\geq 1$ and $\Ext^1_R(\Omega^{-2}_R M, M)=0$. Note that $\Ext^i_R(\Tr_R M, R)$ has finite length for each $i\geq 1$ since $M$ is a vector bundle. Thus $\Hom_R(\Ext^i_R(\Tr_R M, R), M)=0$ for all $i\geq 1$; see \ref{Tor2}(iv). As we assume $\Ext^1_R(\Omega^{-i}_R M, M)=0$ for all $i=2,4, \ldots, 2n$, the exact sequence \ref{Tor2}(v) yields $\Tor_i^R(M, \Tr_R M)=0$ for all $i=2,4, \ldots, 2n$. 

Next assume $\depth_R(M\otimes_RM^{\ast})\geq 1$ and $\Ext^i_R(\Tr_R M, R)=0$ for all $i=4,6, \ldots, 2n$. As we assume $\Ext^1_R(\Omega^{-i}_R M, M)=0$ for all $i=4, \ldots, 2n$, the exact sequence \ref{Tor2}(v) yields the vanishing of $\Tor_i^R(M, \Tr_R M)$ for all $i=4,6, \ldots, 2n$. Note that $\Tor_2^R(M, \Tr_R M)$ has finite length as $M$ is a vector bundle. So \ref{Tor2}(iii) implies that $\Tor_2^R(M, \Tr_R M)=0$ as $\depth_R(M\otimes_RM^{\ast})\geq 1$. Hence we obtain $\Tor_i^R(M, \Tr_R M)=0$ for all $i=2,4, \ldots, 2n$.

(2) Set $d=\dim(R)$. If $d=0$, then $\depth(R)=0$ and so $M$ is free since $M$ is generically free. Due to our assumption in part (a), the localization of a torsion-free module is torsion-free; see \cite[3.8]{YYNE}. Hence, by induction on $d$, we may assume $M$ is a vector bundle and $\depth(R)\geq 1$. Then \ref{Tor2}(iii) implies that $\Tor_2^R(M, \Tr_R M)=0$ since $\depth_R(\Omega_R^2M \otimes_R \Tr_R M)\geq 1$.

Assume $\Ext^1_R(\Tr_R M,R)=0$. Then $M$ is torsionless so that $\depth_R(M)\geq 1$. Moreover \ref{Tor2}(v) yields an injection $\Ext^1_R(\Omega^{-2}_R M, M)  \hookrightarrow \Tor_2^R(\Tr_R M, M)$, which implies that $\Ext^1_R(\Omega^{-2}_R M, M)=0$. Therefore, $M$ is free by part (1) of the theorem.

Next assume $\Ext^i_R(\Tr_R M,R)=0$ for all $i=4,6, \ldots, 2n$. Then \ref{Tor2}(v) shows that $\Tor_i^R(M, \Tr_R M)=0$ for all $i=4,6, \ldots, 2n$
because we assume the vanishing of $\Ext^1_R(\Omega^{-i}_R M, M)=0$ for all $i=4,6, \ldots, 2n$. As $\Tor_2^R(M, \Tr_R M)=0$, we conclude that  $\Tor_i^R(M, \Tr_R M)=0$ for all $i=2, 4,\ldots, 2n$. Consequently, $M$ is free by Lemma \ref{mainlemma}.
\end{proof}

We proceed to prove the extension of \cite[1.3]{Fukano} mentioned in the introduction; see Proposition \ref{corGor}. First we prepare a lemma.

\begin{lem} \label{lemsupp} Assume $\depth(R)\geq 1$ and let $M$ be an $R$-module. Assume $M$ is totally reflexive $R$-module which is $t$-periodic for some $t\geq 2$.  If $\Ext^{t-1}_R(M,M)=0$, then $\depth_R(M\otimes_RM^{\ast})\geq 1$. Also, if $M$ is a vector bundle and $\depth_R(M\otimes_RM^{\ast})\geq 1$, then $\Ext^{t-1}_R(M,M)=0$.
\end{lem}

\begin{proof} Note that $\Ext^{t-1}_R(M,M) \cong \Ext^{1}_R(\Omega_R ^{t-2}M,M) \cong \Ext^{1}_R(\Omega_R ^{-2}M,M)$ since $M$ is $t$-periodic. Note also that $\Ext^{2}_R(\Tr_R M,R)=0$ since $M$ is reflexive. Hence, by letting $i=2$, we obtain from the exact sequence in \ref{Tor2}(v) that $\Ext^{t-1}_R(M,M) \cong \Tor_2^R(\Tr_R M, M)$.

It follows that $\depth_R\big(\Hom_R(M,M)\big)\geq 1$ since $\depth_R(M)=\depth(R)\geq 1$. Hence the claims follow from the exact sequence in \ref{Tor2}(iii).
\end{proof}

\begin{cor} \label{cortr} Let $M$ be a totally reflexive $R$-module which is a vector bundle and $2n$-periodic for some $n\geq 1$. Assume $\depth(R)\geq 1$. Assume further:
\begin{enumerate}[\rm(i)]
\item $[M]=0$ in $\overline{\G}(R)_{\QQ}$.
\item $\depth_R(M\otimes_RM^{\ast})\geq 1$
\item $\Ext^i_R(M,M)=0$ for all $i=1,3, \ldots, 2n-3$.
\end{enumerate}
Then $M$ is free.
\end{cor}

\begin{proof} We have, by Lemma \ref{lemsupp}, that $\Ext^{t-1}_R(M,M)=0$. Hence the claim follows from Corollary \ref{maincor}(1).
\end{proof}

\begin{prop} \label{corGor} Let $R$ be a Gorenstein local ring and let $M$ be a $2$-periodic $R$-module. Assume $M$ is generically free and $[M]=0$ in $\overline{\G}(R)_{\QQ}$. If $M\otimes_RM^{\ast}$ is torsion-free, then $M$ is free.
\end{prop}

\begin{proof} We proceed by induction on the dimension $d$ of $R$. There is nothing to prove if $d=0$. We may assume, by the induction hypothesis, that $d\geq 1$ and $M$ is a vector bundle. Then Corollary \ref{cortr} implies that $M$ is free; see \ref{gglist}(iv).
\end{proof}

Note that Proposition \ref{corGor} follows from Corollary \ref{2-periodic-prop-cor} when $R$ is one-dimensional; see \ref{Tor2}(vi). Note also that the conclusion of the proposition can fail if the module in question is not $2$-periodic:

\begin{eg} [{\cite[3.5]{CeD2}}] Let $R=k[\![x,y,z]\!]/(xy-z^2)$. Then $R$ is a two-dimensional normal domain and $\overline{\G}(R)_{\QQ}=0$; see \ref{gglist}(iii). If $I$ is the ideal of $R$ generated by $x$ and $y$, then $I^{\ast} \cong R$ so that $I \otimes_R I^{\ast} \cong I$ is torsion-free. Moreover, $\pd_R(I)=1$, $I$ is not a periodic $R$-module, and $\Ext^1_R(I,I)\neq 0$.
\end{eg}

\section{Acknowledgements}
Part of this work was completed during the visit of Takahashi to the School of Mathematical and Data Sciences at West Virginia University in May 2023.

The authors thank Brian Laverty and Uyen Le for their comments on previous versions of the manuscript, and Toshinori Kobayashi for discussions and help with Theorem \ref{introFukano}.

\end{document}